\documentclass[11pt]{amsart}
\newtheorem{thm}{Theorem}
\newtheorem{lem}{Lemma}
\newtheorem{prop}{Proposition}
\newtheorem{definition}{Definition}

\newcommand{\R}{\mathbb{R}}

\usepackage{amsmath}
\usepackage{hyperref}
\hypersetup{colorlinks=true, urlcolor=black, linkcolor=black, citecolor=black}

\usepackage{mathrsfs}

\begin{document}
\title[Asymptotics of signed Bernoulli convolutions]
{Asymptotics of signed Bernoulli convolutions scaled by multinacci numbers}
\date{\today}
\author{Xianghong Chen \and Tian-You Hu}
\address{X. Chen\\Department of Mathematical Sciences\\University of Wisconsin-Milwaukee\\Milwaukee, WI 53211, USA}
\email{chen242@uwm.edu}
\address{T. Hu\\Department of Mathematics\\University of Wisconsin-Green Bay\\Wisconsin, WI 54311, USA}
\email{hut@uwgb.edu}


\begin{abstract}
We study the signed Bernoulli convolution 
$$\nu_\beta^{(n)}=*_{j=1}^{n}
\left (\frac{1}{2}\delta_{\beta^{-j}}-\frac{1}{2}\delta_{-\beta^{-j}}\right ),\ n\ge 1$$
where $\beta>1$ satisfies 
$$\beta^m=\beta^{m-1}+\cdots+\beta+1$$
for some integer $m\ge 2$. When $m$ is odd, we show that the variation $|\nu_\beta^{(n)}|$ coincides the unsigned Bernoulli convolution
$$\mu_\beta^{(n)}=*_{j=1}^{n}
\left (\frac{1}{2}\delta_{\beta^{-j}}+\frac{1}{2}\delta_{-\beta^{-j}}\right ).$$
When $m$ is even, we obtain the exact asymptotic of the total variation $\|\nu_\beta^{(n)}\|$ as $n\rightarrow\infty$. 
\end{abstract}
\maketitle

\section{Introduction}
In this paper we initiate the study of the signed Bernoulli convolution
\begin{equation}\label{signed-bernoulli}
\nu_\beta^{(n)}=*_{j=1}^{n}
\left (\frac{1}{2}\delta_{\beta^{-j}}-\frac{1}{2}\delta_{-\beta^{-j}}\right ),\ n\ge 1
\end{equation}
where $\beta>1$ and $\delta_x$ denotes the Dirac measure at $x\in\R$. Equivalently, $\nu_\beta^{(n)}$ is defined inductively by letting 
\begin{equation}\label{inductive}
\begin{cases}
\nu_\beta^{(0)}=\delta_0,\\
\nu_\beta^{(n)}=\nu_\beta^{(n-1)}*\left (\frac{1}{2}\delta_{\beta^{-n}}-\frac{1}{2}\delta_{-\beta^{-n}}\right ),\ n\ge 1.
\end{cases}
\end{equation}
By expanding out the convolution in \eqref{signed-bernoulli}, we also have
\begin{equation}\label{expansion}
\nu_\beta^{(n)}
=\frac{1}{2^n} \sum_{\varepsilon_j=\pm 1\atop j=1,\cdots,n} 
\varepsilon_1\cdots\varepsilon_n \,\delta_{\sum_{j=1}^n \varepsilon_j \beta^{-j}},\ n\ge 1.
\end{equation}
The definition of $\nu_\beta^{(n)}$ is related to that of the unsigned Bernoulli convolution
\begin{align*}
\mu_\beta^{(n)}
=*_{j=1}^{n}
\left (\frac{1}{2}\delta_{\beta^{-j}}+\frac{1}{2}\delta_{-\beta^{-j}}\right )
\end{align*}
which converges weakly to a probability measure $\mu_\beta$ as $n\rightarrow\infty$. The measure $\mu_\beta$ is a classical subject of study. We refer the reader to \cite{PeresSchlagSolomyak2000} for more background, and to \cite{Shmerkin2014}, \cite{Varju2016}, \cite{HareSidorov2016} for some recent results.

In contrast to $\mu_\beta^{(n)}$, the signed Bernoulli convolution $\nu_\beta^{(n)}$ converges weakly to the null measure for any $\beta>1$. This can be seen by writing for $f\in C(\R)$
$$
\langle \nu_\beta^{(n)},f \rangle
=\frac{1}{2}
\langle \nu_\beta^{(n-1)},f*\delta_{-\beta^{-n}}-f*\delta_{\beta^{-n}}\rangle
$$
(where $\langle \sigma,f \rangle$ denotes $\int f d\sigma$), noting that the last integral converges to zero as $\nu_\beta^{(n-1)}$ is supported in $[-(\beta-1)^{-1},(\beta-1)^{-1}]$ and $\beta^{-n}\rightarrow 0$. Moreover, if there is cancellation in the expansion \eqref{expansion}, or equivalently, if the total variation satisfies
$$\|\nu_\beta^{(n)}\|<1$$
for some $n$, then by Young's convolution inequality $\|\nu_\beta^{(n)}\|$ must decay at least exponentially as $n\rightarrow\infty$. It is then of interest to determine the exact rate of decay of $\|\nu_\beta^{(n)}\|$ in such situation.

In the present paper, we study the case where $\beta>1$ satisfies
\begin{equation}\label{evenacci}
\beta^m=\beta^{m-1}+\cdots+\beta+1
\end{equation}
for some integer $m\ge 2$. 
Note that when $m=2$, this corresponds to the golden ratio
\begin{equation}\label{m=2}
\beta=\frac{1+\sqrt{5}}{2}=1.618033\cdots.
\end{equation}
Our main result is the following.

\begin{thm}\label{thm}
Let 
$$a_n=2^n\|\nu_\beta^{(n)}\|,\ n\ge 0.$$
Suppose $\beta$ satisfies \eqref{evenacci} for an even integer $m$. Then
\begin{equation}
a_n=\begin{cases}\label{recurrence}
2^n, &\text{if } n\le m,\\
2a_{n-1}-2a_{n-m}+2a_{n-m-1}, &\text{if } n\ge m+1.
\end{cases}
\end{equation}
In particular, there exists a constant $C>0$ such that
\begin{equation}\label{asymptotic}
\|\nu_\beta^{(n)}\|\sim C\left (\frac{\lambda}{2}\right )^n,\ \text{as }n\rightarrow\infty
\end{equation}
where $\lambda\in(1,2)$ is the only real root of the equation
\begin{equation}\label{characteristic}
x^{m+1}=2x^m-2x+2.
\end{equation}
\end{thm}

For instance, if $\beta$ is given by the golden ratio \eqref{m=2}, Theorem \ref{thm} gives
$$\|\nu_\beta^{(n)}\|\sim C (0.771844\cdots)^n,\ \text{as }n\rightarrow\infty.$$
On the other hand, if $\beta$ satisfies \eqref{evenacci} for an odd integer $m$, then the total variation $\|\nu_\beta^{(n)}\|$ has no decay in $n$; in fact, $\|\nu_\beta^{(n)}\|\equiv 1$ in this case. This will be shown in Section \ref{sec:remarks}.

The proof of Theorem \ref{thm} relies on analyzing the cancellation pattern 
in $\nu_\beta^{(n)}$ as $n$ increases. The presence of overlap (due to $\beta<2$) is remedied by the fact that cancellation occurs whenever an overlap is formed. This allows us to identify $\{\nu_\beta^{(n)}\}_{n\ge 0}$ with a plane tree, based on which the recurrence relation \eqref{recurrence} is derived. The situation becomes more involved if one considers more general $\beta$. However, we hope the analysis in this paper will provide a simple model for the study of more general situations. As an application, Theorem \ref{thm} can be used to derive nontrivial bounds on certain sine products; see Section \ref{sec:remarks}.

Although not directly related, our study should be compared with that of the unsigned Bernoulli convolutions. We refer the reader to \cite{AlexanderZagier1991}, \cite{Lau1992}, \cite{Lau1993}, \cite{LedrappierPorzio1994}, \cite{LedrappierPorzio1996}, \cite{Hu1997}, \cite{Ngai1997}, \cite{LauNgai1998}, \cite{SidorovVershik1998}, \cite{Lalley1998}, \cite{Porzio1998}, \cite{LauNgai1999}, \cite{Fan2002}, \cite{GrabnerKirschenhoferTichy2002}, \cite{Feng2005}, and references therein.

The plan of the paper is as follows. In Section \ref{sec:notation}, we introduce some notation which will be used throughout the remainder of the paper. We also state at the end three main lemmas. In Sections \ref{sec:diamond}, \ref{sec:order} and \ref{sec:leafless}, respectively, we give the proofs of the three lemmas. In Section \ref{sec:proof} we prove Theorem \ref{thm}. In Section \ref{sec:remarks}, we conclude the paper with some remarks.

\section{Notation and definition}\label{sec:notation}
For $n\ge 1$, denote
$$\mathscr D_n=\{\pm 1\}^n.$$
Suppose
$$\boldsymbol\varepsilon=(\varepsilon_1,\cdots,\varepsilon_n)\in \mathscr D_n.$$
We will denote
\begin{equation}\label{eq:x_eps}
x_{\boldsymbol\varepsilon} :=\sum_{j=1}^{n} \varepsilon_j \beta^{-j}\in\R.
\end{equation}
With this notation, we can write
\begin{equation}\label{eq:expansion}
\nu_\beta^{(n)}
=\frac{1}{2^n}\sum_{\boldsymbol\varepsilon\in\mathscr D_n} \varepsilon_1\cdots\varepsilon_n \,\delta_{x_{\boldsymbol\varepsilon}},\ n\ge 1.
\end{equation}
Note that some of the ${x_{\boldsymbol\varepsilon}}$'s may coincide and cancel each other. This motivates the definition
$$A_n:=\text{supp}(\nu_\beta^{(n)})\subset\R,\ n\ge 0.$$
Where $\text{supp}(\cdot)$ stands for the support of the measure. It will be crucial to understand the sets $A_n$. To this end, we introduce more notation. 

Suppose $\boldsymbol\varepsilon$ is as above and 
$$\boldsymbol\varepsilon'=(\varepsilon'_1,\cdots,\varepsilon'_{n+1})\in \mathscr D_{n+1}.$$
We write
$$\boldsymbol\varepsilon\rightarrow\boldsymbol\varepsilon'$$
if 
$$\varepsilon_j=\varepsilon'_j,\ j=1,\cdots,n,$$
in which case we call $\boldsymbol\varepsilon'$ a \textit{child} of $\boldsymbol\varepsilon$ (\cite[Section 3.2]{GrossYellen2006}). Note that $\boldsymbol\varepsilon'$ must be one of
$$\boldsymbol\varepsilon-:=(\varepsilon_1,\cdots,\varepsilon_n,-1),$$
$$\boldsymbol\varepsilon+:=(\varepsilon_1,\cdots,\varepsilon_n,+1).$$
For convenience, we let
$$\mathscr D_0=\{0\},\quad x_0=0$$
and set 
$$0\rightarrow{-1},\quad 0\rightarrow{+1}.$$

\begin{definition}
Using the above notation, we define inductively 
$$\quad \begin{cases}
\mathscr D_0^*=\mathscr D_0,\\
\mathscr D_{n+1}^*=\{\boldsymbol\varepsilon'\in \mathscr D_{n+1}:\ 
\boldsymbol\varepsilon\rightarrow\boldsymbol\varepsilon' \text{ for some } \boldsymbol\varepsilon\in\mathscr D_{n}^* 
\text{ and } x_{\boldsymbol\varepsilon'}\in A_{n+1}\},\ n\ge 0.
\end{cases}$$
\end{definition}

In general, the set $\mathscr D_n^*$ is represented by $A_n$ with multiplicities. However, under the assumption of Theorem \ref{thm}, it will be shown that the multiplicity is always one, namely, $\mathscr D_n^*$ is `isomorphic' to $A_n$ (see Lemma \ref{lem:order-preserve} below).

Note that, together with the relation ``$\rightarrow$'', the set
$$\mathscr T:=\bigcup_{n\ge 0} \mathscr D_n$$
forms a \textit{directed rooted tree} (\cite[Section 3.2]{GrossYellen2006}). Moreover, each $\mathscr D_n$ is naturally equipped with the \textit{lexicographical order} (with the convention $-1<+1$). Let
$$\mathscr T^*:=\bigcup_{n\ge 0} \mathscr D_n^*.$$
Then $\mathscr T^*$ can be thought of as obtained from pruning the full binary tree $\mathscr T$ defined above. As a subset of $\mathscr T$, $\mathscr T^*$ inherits the relation ``$\rightarrow$'' and itself becomes a directed rooted tree (with the same root $0$); also, each $\mathscr D_n^*$ inherits the lexicographical order of $\mathscr D_n$. 

For $\boldsymbol\varepsilon\in\mathscr T$, we will denote by $\mathscr T(\boldsymbol\varepsilon)$ the \textit{subtree} (\cite[Section 3.1]{GrossYellen2006}) of $\mathscr T$ rooted at $\boldsymbol\varepsilon$. Similarly, for $\boldsymbol\varepsilon^*\in\mathscr T^*$, $\mathscr T^*(\boldsymbol\varepsilon^*)$ denotes the subtree of $\mathscr T^*$ rooted at $\boldsymbol\varepsilon^*$. For $n\ge 0$, $\mathscr T^*(\boldsymbol\varepsilon^*;n)$ denotes the \textit{$n$-th level} (\cite[Section 3.1]{GrossYellen2006}) of $\mathscr T^*(\boldsymbol\varepsilon^*)$, with the convention $\mathscr T^*(\boldsymbol\varepsilon^*;0)=\{\boldsymbol\varepsilon^*\}$.

The proof of Theorem \ref{thm} will be based on the following three lemmas, whose proofs are given in Sections \ref{sec:diamond}, \ref{sec:order} and \ref{sec:leafless} respectively. 

From now on till the end of Section \ref{sec:proof}, unless otherwise stated, we will always assume that $\beta$ satisfies \eqref{evenacci} for an even integer $m$, as assumed in Theorem \ref{thm}.

\begin{lem}[The first pruning]\label{lem:diamond}
$$\mathscr D_{n}^*=
\begin{cases} 
\mathscr D_{n}, & \text{if } n\le m,\\
\mathscr D_{n}\backslash\{(-,+,\cdots,+),(+,-,\cdots,-)\}, & \text{if } n=m+1.
\end{cases}$$
\end{lem}

\begin{lem}[Tree isomorphism]\label{lem:order-preserve}
For any $n\ge 0$, the map
\begin{align*}
x_{\boldsymbol\cdot}: 
\mathscr D_{n}^*\rightarrow A_n,\ 
\boldsymbol\varepsilon^*\mapsto x_{\boldsymbol\varepsilon^*}
\end{align*}
is bijective and order-preserving, that is,
$$\boldsymbol\varepsilon^*_a<\boldsymbol\varepsilon^*_b\ \Rightarrow\ 
x_{\boldsymbol\varepsilon^*_a}<x_{\boldsymbol\varepsilon^*_b}.$$ 
\end{lem}

\begin{lem}[Leaflessness]\label{lem:leafless}
Each $\boldsymbol\varepsilon^*\in \mathscr T^*$ has at least one child in $\mathscr T^*$.
\end{lem}

\section{Proof of Lemma \ref{lem:diamond}}\label{sec:diamond}
Lemma \ref{lem:diamond} follows momentarily from the following lemma.

\begin{lem}\label{lem:top-pyramid}
Suppose $\boldsymbol\varepsilon_a,\boldsymbol\varepsilon_b\in\mathscr D_{n}$.\\
(i) If $n\le m$, then
$$\boldsymbol\varepsilon_a<\boldsymbol\varepsilon_b
\ \Rightarrow\ x_{\boldsymbol\varepsilon_a}<x_{\boldsymbol\varepsilon_b}.$$
(ii) If $n=m+1$, then
$$\boldsymbol\varepsilon_a<\boldsymbol\varepsilon_b
\ \Rightarrow\ x_{\boldsymbol\varepsilon_a}\le x_{\boldsymbol\varepsilon_b}$$
and equality holds exactly when
\begin{align*}
\boldsymbol\varepsilon_a=(-,+,\cdots,+),\quad 
\boldsymbol\varepsilon_b=(+,-,\cdots,-).
\end{align*}
\end{lem}

\begin{proof}[Lemma \ref{lem:top-pyramid} $\Rightarrow$ Lemma \ref{lem:diamond}]
Combining \eqref{eq:expansion} and part (i) of Lemma \ref{lem:top-pyramid}, we see that $|A_n|=2^n$ holds when $n\le m$. Therefore, by the definition of $\mathscr D_{n}^*$,
$$\mathscr D_{n}^*=\mathscr D_{n}$$
for all $n\le m$. If $n=m+1$, then by part (ii) of Lemma \ref{lem:top-pyramid}, the map
$$x_{\boldsymbol\cdot}: 
\mathscr D_{n}\backslash\{(-,+,\cdots,+)\}\rightarrow\R$$
is injective; moreover, since $m$ is even, in \eqref{eq:expansion} the Dirac measure at $x_{(-,+,\cdots,+)}$ cancels that at $x_{(+,-,\cdots,-)}$ due to overlap and opposite signs of coefficients. From this it follows that
$$|A_{n}|=2^n-1$$
and 
$$\mathscr D_{n}^*=\mathscr D_{n}\backslash\{(-,+,\cdots,+),(+,-,\cdots,-)\}.$$
This proves Lemma \ref{lem:diamond} assuming the truth of Lemma \ref{lem:top-pyramid}.
\end{proof}

It remains to prove Lemma \ref{lem:top-pyramid}. The proof is based on the following lemma. For convenience, we will write
$$\rho=\beta^{-1}.$$

\begin{lem}\label{lem:1st-collision}
Suppose $\beta>1$ satisfies \eqref{evenacci} for some integer $m\ge 2$. Then for any $n\le m$, we have
$$\rho-\sum_{j=2}^n \rho^{j}\ge \rho^{n+1}.$$
\end{lem}

\begin{proof}
By \eqref{evenacci}, we have
$$1-\rho-\cdots-\rho^m=0.$$
Therefore, if $n\le m$, then
$$1-\rho-\cdots-\rho^{n-1}=\rho^{n}+\cdots+\rho^m\ge \rho^n.$$
Multiplying both sides by $\rho$, we obtain the desired bound.
\end{proof}

We can now prove Lemma \ref{lem:top-pyramid}. 

\begin{proof}[Proof of Lemma \ref{lem:top-pyramid}] Write 
$$\boldsymbol\varepsilon_b-\boldsymbol\varepsilon_a=2(\eta_1,\cdots,\eta_n)$$
where each $\eta_j\in\{\pm 1,0\}$. Suppose $\eta_{j_1}$ is the first nonzero component of $(\eta_1,\cdots,\eta_n)$, then, since $\boldsymbol\varepsilon_a<\boldsymbol\varepsilon_b$, we must have
$$\eta_{j_1}=1.$$ 
Correspondingly,
\begin{align*}
x_{\boldsymbol\varepsilon_b}-x_{\boldsymbol\varepsilon_a}
&=2\sum_{j=j_1}^n \eta_j \rho^{j}\\
&=2\Big (\rho^{j_1}+\sum_{j=j_1+1}^n \eta_j \rho^{j}\Big ).
\end{align*}
(i) If $n\le m$, then by Lemma \ref{lem:1st-collision}, 
\begin{align*}
\rho^{j_1}+\sum_{j=j_1+1}^n \eta_j \rho^{j}
&\ge \rho^{j_1}-\sum_{j=j_1+1}^n \rho^{j}\\
&\ge \rho^{n+1}.
\end{align*}
Thus $x_{\boldsymbol\varepsilon_b}-x_{\boldsymbol\varepsilon_a}>0$.\\
(ii) If $n=m+1$ and $j_1>1$, then the same argument shows that $x_{\boldsymbol\varepsilon_b}-x_{\boldsymbol\varepsilon_a}>0$. If $n=m+1$ and $j_1=1$, then 
\begin{align*}
\rho^{j_1}+\sum_{j=j_1+1}^{n} \eta_j \rho^{j}
&=\rho+\sum_{j=2}^{m+1} \eta_j \rho^{j}\\
&\ge \rho-\sum_{j=2}^{m+1} \rho^{j}\\
&=0.
\end{align*}
Moreover, the inequality above is strict unless $\eta_2=\cdots=\eta_n=-1$, which corresponds to the case 
\begin{align*}
\boldsymbol\varepsilon_a=(-,+,\cdots,+),\quad 
\boldsymbol\varepsilon_b=(+,-,\cdots,-).
\end{align*}
Therefore, except for this case (where $x_{\boldsymbol\varepsilon_a}=x_{\boldsymbol\varepsilon_b}$) we always have $x_{\boldsymbol\varepsilon_a}<x_{\boldsymbol\varepsilon_b}$. This completes the proof of Lemma \ref{lem:top-pyramid}.
\end{proof}

\section{Proof of Lemma \ref{lem:order-preserve}}\label{sec:order}
To prove Lemma \ref{lem:order-preserve}, we will use the following.

\begin{lem}\label{lem:tail-gap}
Suppose $\beta>1$ satisfies \eqref{evenacci} for some integer $m\ge 2$. Then for any $n\ge 0$, we have
$$\sum_{j=n+1}^\infty \rho^{j}< 2\rho^{n}.$$
\end{lem}

\begin{proof}
It suffices to consider the case $n=0$. We need to show
$$\sum_{j=1}^\infty \rho^{j}=\frac{\rho}{1-\rho}<2.$$
However, by \eqref{evenacci},
$$1=\rho+\cdots+\rho^m=\frac{\rho}{1-\rho}(1-\rho^m).$$
Therefore, it suffices to show $\rho^m<{1}/{2}$, or equivalently, $\beta^m>2$. But this follows immediately from \eqref{evenacci} and the assumption $\beta>1$.
\end{proof}

Lemma \ref{lem:tail-gap} implies the following. 

\begin{lem}\label{lem:non-collision-1}
For any $n\ge 1$ and $\boldsymbol\varepsilon_a^*, \boldsymbol\varepsilon_b^*\in\mathscr D_{n}^*$,
$$\boldsymbol\varepsilon_a^*\in\mathscr T^*(-),\ \boldsymbol\varepsilon_b^*\in\mathscr T^*(+)
\ \Rightarrow\ x_{\boldsymbol\varepsilon_a^*+}\le x_{\boldsymbol\varepsilon_b^*-}$$	
and equality holds if and only if $n=m$ and
$$\boldsymbol\varepsilon_a^*=(-,+,\cdots,+),\quad 
\boldsymbol\varepsilon_b^*=(+,-,\cdots,-).$$
\end{lem}

\begin{proof}
By Lemma \ref{lem:diamond} and Lemma \ref{lem:top-pyramid}, the statement holds for $n\le m$. Suppose $n\ge m+1$. Then there exist 
$$\boldsymbol\varepsilon_{a_\circ}^*\in\mathscr T^*(-;m),\quad  \boldsymbol\varepsilon_{b_\circ}^*\in\mathscr T^*(+;m)$$
such that
$$\boldsymbol\varepsilon_{a}^*+\,\in\mathscr T(\boldsymbol\varepsilon_{a_\circ}^*),\quad  \boldsymbol\varepsilon_{b}^*-\,\in\mathscr T(\boldsymbol\varepsilon_{b_\circ}^*).$$
It follows from \eqref{eq:x_eps} and Lemma \ref{lem:tail-gap} that
\begin{align*}
|x_{\boldsymbol\varepsilon_{a}^*+}-x_{\boldsymbol\varepsilon_{a_\circ}^*}|,\  |x_{\boldsymbol\varepsilon_{b}^*-}-x_{\boldsymbol\varepsilon_{b_\circ}^*}|
< \sum_{j=m+2}^{\infty} \rho^{j}
< 2\rho^{m+1}.
\end{align*}
On the other hand, by Lemma \ref{lem:diamond},
$$(-,+,\cdots,+)\notin\mathscr T^*(-;m),\quad 
(+,-,\cdots,-)\notin\mathscr T^*(+;m);$$
hence
\begin{align*}
x_{\boldsymbol\varepsilon_{b_\circ}^*}-x_{\boldsymbol\varepsilon_{a_\circ}^*}
&\ge x_{\boldsymbol\varepsilon_{(+,-,\cdots,-,+)}^*}-x_{\boldsymbol\varepsilon_{(-,+,\cdots,+,-)}^*}\\
&=2\Big (\rho-\sum_{j=2}^{m} \rho^{j}\Big )+2 \rho^{m+1}\\
&=4 \rho^{m+1}.
\end{align*}
Combining these, we get
\begin{align*}
x_{\boldsymbol\varepsilon_{b}^*-}-x_{\boldsymbol\varepsilon_{a}^*+}
&=(x_{\boldsymbol\varepsilon_{b}^*-}-x_{\boldsymbol\varepsilon_{b_\circ}^*})
+(x_{\boldsymbol\varepsilon_{b_\circ}^*}-x_{\boldsymbol\varepsilon_{a_\circ}^*})
-(x_{\boldsymbol\varepsilon_{a}^*+}-x_{\boldsymbol\varepsilon_{a_\circ}^*})\\
&> - 2\rho^{m+1} + 4 \rho^{m+1} - 2\rho^{m+1}\\
&=0.
\end{align*}
This shows $x_{\boldsymbol\varepsilon_a^*+}<x_{\boldsymbol\varepsilon_b^*-}$ whenever $n\ge m+1$, and the proof is complete.
\end{proof}

Based on Lemma \ref{lem:non-collision-1}, we can now prove: 

\begin{lem}[A separation property]\label{lem:non-collision-2}
Suppose $n\ge 1$ and $\boldsymbol\varepsilon_a^*,\boldsymbol\varepsilon_b^*\in\mathscr D_{n}^*$ satisfy 
$$\boldsymbol\varepsilon_a^*<\boldsymbol\varepsilon_b^*.$$
Then for any $k\ge 0$,
$$\boldsymbol\varepsilon_{a'}^*\in\mathscr T^*(\boldsymbol\varepsilon_a^*;k),\ \boldsymbol\varepsilon_{b'}^*\in\mathscr T^*(\boldsymbol\varepsilon_b^*;k)
\ \Rightarrow\ x_{\boldsymbol\varepsilon_{a'}^*+}\le x_{\boldsymbol\varepsilon_{b'}^*-}$$
and strict inequality holds when $k\ge m$. 
\end{lem}

\begin{proof}
The proof is by induction on $n$. The case $n=1$ follows directly from Lemma \ref{lem:non-collision-1}. Suppose the statement holds for $1,\cdots,n-1$. We now prove that it also holds for $n$.

It suffices to consider the case where
$$\boldsymbol\varepsilon_a^*=\boldsymbol\varepsilon_\circ^*-,\quad 
\boldsymbol\varepsilon_b^*=\boldsymbol\varepsilon_\circ^*+$$
for some $\boldsymbol\varepsilon_\circ^*\in \mathscr D_{n-1}^*$. This is because otherwise we can consider the nearest common ancestor (say $\boldsymbol\varepsilon_\circ^*$) of $\boldsymbol\varepsilon_a^*$ and $\boldsymbol\varepsilon_b^*$, and apply the induction hypothesis to $\boldsymbol\varepsilon_\circ^*-$ and $\boldsymbol\varepsilon_\circ^*+$. 

By the induction hypothesis, for any $\boldsymbol\varepsilon_1^*,\boldsymbol\varepsilon_2^*\in\mathscr D_{n-1}^*$, 
$$\boldsymbol\varepsilon_1^*<\boldsymbol\varepsilon_2^*\Rightarrow x_{\boldsymbol\varepsilon_{1}^*+}\le x_{\boldsymbol\varepsilon_{2}^*-}$$
and equality holds only if
$$\boldsymbol\varepsilon_{1}^*+,\ \boldsymbol\varepsilon_{2}^*-\notin \mathscr D_{n}^*.$$
Since $\boldsymbol\varepsilon_a^*=\boldsymbol\varepsilon_\circ^*-,\ \boldsymbol\varepsilon_b^*=\boldsymbol\varepsilon_\circ^*+\in\mathscr D_{n}^*$, after mapped by $x_\cdot$, $\boldsymbol\varepsilon_a^*$ and $\boldsymbol\varepsilon_b^*$ do not overlap with the (either left or right) children of the nodes in $\mathscr D_{n-1}^*\backslash \{\boldsymbol\varepsilon_\circ^*\}$. Consequently, we have
$$\boldsymbol\varepsilon_\circ^*-+\cdots+\,\in \mathscr T^*(\boldsymbol\varepsilon_\circ^*;m),$$
$$\boldsymbol\varepsilon_\circ^*+-\cdots-\,\in \mathscr T^*(\boldsymbol\varepsilon_\circ^*;m),$$
and by cancellation
$$\boldsymbol\varepsilon_\circ^*-+\cdots+\,\notin \mathscr T^*(\boldsymbol\varepsilon_\circ^*;m+1),$$
$$\boldsymbol\varepsilon_\circ^*+-\cdots-\,\notin \mathscr T^*(\boldsymbol\varepsilon_\circ^*;m+1).$$
Therefore, by the same argument as in the proof of Lemma \ref{lem:non-collision-1}, we have, for any $k\ge m$, 
$$\boldsymbol\varepsilon_{a'}^*\in\mathscr T^*(\boldsymbol\varepsilon_\circ^*-;k),\ \boldsymbol\varepsilon_{b'}^*\in\mathscr T^*(\boldsymbol\varepsilon_\circ^*+;k)
\ \Rightarrow\ x_{\boldsymbol\varepsilon_{a'}^*+}<x_{\boldsymbol\varepsilon_{b'}^*-}.$$	
The case $k\le m-1$ is obvious. This completes the proof.
\end{proof}

Lemma \ref{lem:order-preserve} now follows immediately.

\begin{proof}[Proof of Lemma \ref{lem:order-preserve}]
By taking $k=0$ in Lemma \ref{lem:non-collision-2}, we see that
$$\boldsymbol\varepsilon^*_a<\boldsymbol\varepsilon^*_b\ \Rightarrow\ 
x_{\boldsymbol\varepsilon^*_a}<x_{\boldsymbol\varepsilon^*_a+}\le x_{\boldsymbol\varepsilon^*_b-}<x_{\boldsymbol\varepsilon^*_b}.$$ 
In particular, the map
$$x_{\boldsymbol\cdot}: 
\mathscr D_{n}^*\rightarrow A_n,\ 
\boldsymbol\varepsilon^*\mapsto x_{\boldsymbol\varepsilon^*}$$
is order-preserving (thus injective). Surjectivity of this map follows easily from \eqref{inductive} and induction on $n$. 
\end{proof}	

As a corollary of Lemma \ref{lem:order-preserve}, the following identities follow easily by induction.

\begin{lem}\label{lem:cor}
For any $n\ge 1$, we have
\begin{equation}\label{eq:expansion*}
\nu_\beta^{(n)}
=\frac{1}{2^n}\sum_{\boldsymbol\varepsilon\in\mathscr D_n^*} \varepsilon_1\cdots\varepsilon_n \,\delta_{x_{\boldsymbol\varepsilon}}.
\end{equation}
In particular,
\begin{equation}\label{eq:variation}
\|\nu_\beta^{(n)}\|=\frac{|\mathscr D_n^*|}{2^n},\ n\ge 0.
\end{equation}
\end{lem}

\section{Proof of Lemma \ref{lem:leafless}}\label{sec:leafless}
To prove Lemma \ref{lem:leafless}, we first show:
\begin{lem}[Diamond pattern]\label{lem:one-child}
Suppose $n\ge m$ and $\boldsymbol\varepsilon^*\in\mathscr D_{n}^*$. Then 
$$\boldsymbol\varepsilon^*-\,\notin\mathscr D_{n+1}^*$$
if and only if
$$\boldsymbol\varepsilon^*
=\boldsymbol\varepsilon_\circ^*+-\cdots-$$
for some $\boldsymbol\varepsilon_\circ^*\in \mathscr D_{n-m}^*$ with  $\boldsymbol\varepsilon_\circ^*\pm\,\in \mathscr D_{n-m+1}^*$; similarly, 
$$\boldsymbol\varepsilon^*+\,\notin\mathscr D_{n+1}^*$$
if and only if
$$\boldsymbol\varepsilon^*
=\boldsymbol\varepsilon_\circ^*-+\cdots+$$
for some $\boldsymbol\varepsilon_\circ^*\in \mathscr D_{n-m}^*$ with $\boldsymbol\varepsilon_\circ^*\pm\,\in \mathscr D_{n-m+1}^*$. 
\end{lem}

\begin{proof}
We prove only the first case where $\boldsymbol\varepsilon^*-\,\notin\mathscr D_{n+1}^*$. The other case follows by symmetry. Also, the `if' part is easy by the proof of Lemma \ref{lem:non-collision-2}. So we only need to prove the `only if' part of the statement. 

Suppose $\boldsymbol\varepsilon_\circ^*\,\in \mathscr D_{n-m}^*$ is the $m$-th ancestor of $\boldsymbol\varepsilon^*$, that is,
$$\boldsymbol\varepsilon^*\in\mathscr T^*(\boldsymbol\varepsilon_\circ^*;m).$$ 
Then, by Lemma \ref{lem:non-collision-2}, $\boldsymbol\varepsilon^*-$ must be canceled within $\mathscr T^*(\boldsymbol\varepsilon_\circ^*)$ -- more precisely, there must exist $\boldsymbol\varepsilon_\star^*\in \mathscr T^*(\boldsymbol\varepsilon_\circ^*;m)$
such that
$$x_{\boldsymbol\varepsilon_\star^*+}=x_{\boldsymbol\varepsilon^*-}.$$
However, by Lemma \ref{lem:diamond}, this is impossible unless 
$$\boldsymbol\varepsilon^*=\boldsymbol\varepsilon_\circ^*+-\cdots-,\quad 
\boldsymbol\varepsilon_\star^*=\boldsymbol\varepsilon_\circ^*-+\cdots+.$$
It follows also that 
$$\boldsymbol\varepsilon_\circ^*\pm\,\in \mathscr D_{n-m+1}^*.$$ This completes the proof.
\end{proof}

Lemma \ref{lem:leafless} now follows momentarily.

\begin{proof}[Proof of Lemma \ref{lem:leafless}]
Assume to the contrary that there exists $\boldsymbol\varepsilon^*\in \mathscr D_{n}^*$ with
$$\boldsymbol\varepsilon^*-\,\notin \mathscr D_{n+1}^*,\quad 
\boldsymbol\varepsilon^*+\,\notin \mathscr D_{n+1}^*.$$
Then, by Lemma \ref{lem:one-child}, there must exist $\boldsymbol\varepsilon_a^*, \boldsymbol\varepsilon_b^*\in \mathscr D_{n-1}^*$ such that
$$\boldsymbol\varepsilon_a^*+=\boldsymbol\varepsilon^*=\boldsymbol\varepsilon_b^*-.$$
However, this implies that $\boldsymbol\varepsilon^*\notin \mathscr D_{n}^*$, a contradiction.
\end{proof}

\section{Proof of Theorem \ref{thm}}\label{sec:proof}
We are now ready to prove Theorem \ref{thm}. By Lemma \ref{lem:cor},
$$a_n=2^n\|\nu_\beta^{(n)}\|=|\mathscr D_{n}^*|,\ n\ge 0.$$
In particular, Lemma \ref{lem:diamond} gives
\begin{equation}\label{eq:ini-cond}
a_n=2^n,\ n\le m.
\end{equation}
For $n\ge m$, we have, by the proof of Lemma \ref{lem:non-collision-2},
$$a_{n+1}=2a_{n} - 2 b_n$$
where $b_n$ denotes the number of pairs $(\boldsymbol\varepsilon^*_a,\boldsymbol\varepsilon^*_b)\in \mathscr D_{n}^*\times\mathscr D_{n}^*$ such that 
$$x_{\boldsymbol\varepsilon^*_a+}=x_{\boldsymbol\varepsilon^*_b-}.$$
By Lemma \ref{lem:one-child}, such pairs are in one-to-one correspondence with the nodes $\boldsymbol\varepsilon_\circ^*\in\mathscr D_{n-m}^*$ satisfying 
\begin{equation}\label{eq:2-children}
\boldsymbol\varepsilon_\circ^*\pm\,\in \mathscr D_{n-m+1}^*.
\end{equation}
Note that each of these nodes contributes an increment of one from $a_{n-m}$ to $a_{n-m+1}$. On the other hand, by Lemma \ref{lem:leafless}, all other nodes in $\mathscr D_{n-m}^*$ have exactly one child in $\mathscr D_{n-m+1}^*$, therefore contribute no increment from $a_{n-m}$ to $a_{n-m+1}$. It follows that the number of nodes satisfying \eqref{eq:2-children} is given by 
$a_{n-m+1}-a_{n-m},$
that is,
$$b_n=a_{n-m+1}-a_{n-m}.$$
Combining, we obtain the desired recurrence relation
$$a_{n+1}=2a_{n} - 2 a_{n-m+1} + 2 a_{n-m},\ n\ge m.$$
This completes the proof of \eqref{recurrence}. 

It remains to prove the asymptotic \eqref{asymptotic}. For that we will need:

\begin{lem}\label{lem:roots-distribution}
Let $m\ge 2$ be an even integer. Then the equation
\begin{equation}\label{eq:characteristic-polynomial-z}
z^{m+1}=2 z^m-2z+2
\end{equation}
has exactly one real root; moreover, the real root lies in the interval $(1,2)$ and has the largest absolute value among all the roots of \eqref{eq:characteristic-polynomial-z}.
\end{lem}

\begin{proof}
Let
$$f(z)=z^{m+1}-2z^m+2z-2.$$
By writing
$$f(x)=x^m(x-2)+2(x-1),$$
it is easy to see that $f(x)<0$ when $x\le 1$ and $f(x)>0$ when $x\ge 2$. In particular, $f$ has at least one zero in $(1,2)$. Suppose $m\ge 4$. Then, since
\begin{align*}
f'(x)&=(m+1)x^m-2mx^{m-1}+2,\\
f''(x)&=(m+1)m x^{m-1}-2m(m-1)x^{m-2}
\end{align*}
are both negative at $x=1$, to show that $f$ has only one zero in $(1,2)$ it suffices to show that $f$ has exactly one inflection point in $(1,2)$. However, this is clear as one can write
$$f''(x)=m x^{m-2}\big((m+1)x-2(m-1)\big),$$
which changes sign only at $x=\frac{2(m-1)}{m+1}.$ 
In particular, the zero must lie in the interval
$$\left(\frac{2(m-1)}{m+1},2 \right ).$$
The case $m=2$ is simpler, as $f'$ would be positive on $(1,2)$ in this case. 

Let $\lambda_m$ denote the real root of $f$, and let $z_1,\cdots,z_m$ be the complex roots of $f$. It remains to show
$$|z_j|<\lambda_m,\ j=1,\cdots,m.$$
In fact, we will show
\begin{equation}\label{eq:spectral-radius}
|z_j|<\frac{3}{2},\ j=1,\cdots,m.
\end{equation}
This suffices because $\lambda_m$ is increasing in $m$ and $\lambda_2=1.543\cdots$. To show \eqref{eq:spectral-radius}, we write
$$f(z)=g(z)+h(z)$$	
with
$$g(z)=-2z^m,\quad h(z)=z^{m+1}+2z-2.$$
Let $r\in (1,2)$. Then on the circle $|z|=r$, we have
$$|g(z)|= 2 r^{m},\quad |h(z)|\le r^{m+1}+2r+2.$$
In particular,
$$|h(z)|<|g(z)|,\ |z|=r$$
holds provided
$$r^{m+1}+2r+2<2 r^{m},$$
or equivalently,
\begin{equation}\label{eq:spectral-radius-bound}
r^m>\frac{2r+2}{2-r}.
\end{equation}
Now fix $r=\frac{3}{2}$ and let $m\ge 6$. It is easy to see that \eqref{eq:spectral-radius-bound} is satisfied. By Rouch\'e's theorem, $f(z)$ and $g(z)$ have the same number ($m$) of zeros in the disk $|z|<r$. On the other hand, since $\lambda_m>\frac{3}{2}$, it follows that all the $m$ complex roots of $f$ are in the disk $|z|<r$, that is, \eqref{eq:spectral-radius} holds. 

By direct checking, \eqref{eq:spectral-radius} also holds when $m=2, 4$. This completes the proof of the lemma.
\end{proof}

Now denote by $\lambda$ the real root of \eqref{eq:characteristic-polynomial-z}. Consider the generating function
$$F(z)=\sum_{n=0}^\infty a_n z^n,\ |z|<{\lambda^{-1}}.$$
By \eqref{recurrence}, it is easy to find
$$F(z)=\frac{1+2z^m}{1-2z+2z^m-2z^{m+1}}.$$
Notice that $1+2z^m\neq 0$ when $z=\lambda^{-1}$. Combining this with Lemma \ref{lem:roots-distribution}, by \cite[Theorem~10.8]{Odlyzko1995}, it follows that there exists a constant $C>0$ such that
$$a_n\sim C\lambda^n,\ \text{as }n\rightarrow\infty.$$
This proves \eqref{asymptotic} and the proof of Theorem \ref{thm} is complete.

\section{Remarks}\label{sec:remarks}

\noindent\textit{Remark 1.} By \eqref{expansion}, in order for
$$\|\nu^{(n)}_\beta\|<1$$
to hold, there must exist $(\varepsilon_1',\cdots,\varepsilon_n'),(\varepsilon_1'',\cdots,\varepsilon_n'')\in\{\pm 1\}^n$ such that
\begin{equation}\label{eq:overlap-1}
\sum_{j=1}^n \varepsilon_j' \beta^{-j}
=\sum_{j=1}^n \varepsilon_j'' \beta^{-j}
\end{equation}
and such that
$$\varepsilon_1'\cdots\varepsilon_n'=-\varepsilon_1''\cdots\varepsilon_n''.$$
Rewriting \eqref{eq:overlap-1}, this means
\begin{equation}\label{eq:overlap-2}
2 \sum_{j=1}^n \eta_j \beta^{-j}=0
\end{equation}
holds for some $(\eta_1,\cdots,\eta_n)\in\{0,\pm 1\}^n$ satisfying
$$|\eta_1|+\cdots+|\eta_n|\text{ is odd},$$
or equivalently,
$$\eta_1+\cdots+\eta_n\text{ is odd}.$$
It is easy to see that the above reasoning is revertible. Thus, after multiplying \eqref{eq:overlap-2} by $\beta^n$, we obtain the following.

\begin{lem}\label{lem:zero-element}
$\|\nu^{(n)}_\beta\|<1$ holds if and only if there exists a polynomial
$$p(x)=\sum_{j=0}^{n-1} \eta_j x^j$$
with $(\eta_0,\cdots,\eta_{n-1})\in\{0,\pm 1\}^n$ such that $p(1)$ is odd and $p(\beta)=0$.
\end{lem}

Now suppose $\beta>1$ satisfies \eqref{evenacci} for some odd integer $m\ge 3$. Denote its minimal polynomial by
$$m(x)=x^m-x^{m-1}-\cdots-x-1.$$
If there is an integer-coefficient polynomial
$$p(x)=\sum_{j=0}^{n-1} \eta_j x^j$$
such that $p(\beta)=0$, then, by the minimality of $m(x)$, there must exist $q(x)\in\mathbb Z[x]$, such that
$$p(x)=m(x)q(x).$$
In particular, we have
$$p(1)=m(1)q(1).$$
However, since $m(1)=-(m-1)$ is even, it follows that $p(1)$ must be even too. Thus, combining this with Lemma \ref{lem:zero-element}, we obtain the following.

\begin{prop}\label{prop:oddnacci:no-decay}
Suppose $\beta>1$ satisfies \eqref{evenacci} for an odd integer $m\ge 3$. Then $$\|\nu^{(n)}_\beta\|=1,\ n\ge 1.$$
\end{prop}

\noindent\textit{Remark 2.} By taking the Fourier transform of $\nu_\beta^{(n)}$, one obtains
$$F_n(\beta;\xi):=\prod_{j=1}^n \sin(2\pi \beta^{-j}\xi),\ \xi\in\R.$$
Since the Fourier transform satisfies $\|\widehat{\nu}\|_\infty\le \|\nu\|$, this provides an upper bound for  $\|F_n(\beta;\cdot)\|_\infty$. In particular, when $\beta=\frac{1+\sqrt{5}}{2}$, Theorem \ref{thm} gives
$$\|F_n\|_\infty\le C (0.771844\cdots)^n.$$
Sharpness of this bound will be addressed in \cite{ChenHu}.

\bibliographystyle{abbrv}
\bibliography{bibliography}
\newpage
\end{document}